\documentclass[a4paper,10pt]{amsart}
\usepackage[utf8]{inputenc}

\usepackage{amssymb}
\usepackage{amsthm}
\usepackage{amsmath,amscd}
\usepackage[mathscr]{euscript}
\usepackage[all]{xy}
\usepackage{lmodern}
\usepackage[T1]{fontenc}
\usepackage[textwidth=14cm,hcentering]{geometry}
\usepackage[colorlinks=true,linkcolor=red,citecolor=blue]{hyperref}
\usepackage{cleveref}
\usepackage{mathtools}
\usepackage{xspace}
\usepackage{tikz}
\usetikzlibrary{cd}
\usetikzlibrary{arrows,positioning}

\newtheorem{thm}{Theorem}[section]
\newtheorem{lm}[thm]{Lemma}
\newtheorem{prop}[thm]{Proposition}
\newtheorem{cor}[thm]{Corollary}

\newtheorem*{thm*}{Theorem}

\theoremstyle{definition}
\newtheorem{defn}[thm]{Definition}

\newtheorem{rmk}[thm]{Remark}

\newtheorem*{rmk*}{Remark}

\newtheorem{thmx}{Theorem}

\newtheorem{corx}[thmx]{Corollary}

\newcommand{\op}{^{\mathrm{op}}}

\newcommand{\Pic}{\mathrm{Pic}}
\newcommand{\cat}{\mathbf}

\newcommand{\N}{\mathbb N}
\newcommand{\Z}{\mathbb Z}

\newcommand{\Sch}{\cat{Sch}}

\newcommand{\K}{\mathrm{K}}

\newcommand{\Vect}{\mathrm{Vect}}
\newcommand{\D}{\mathrm{D}}

\newcommand{\Spec}{\mathrm{Spec\,}}
\newcommand{\gr}{\mathrm{gr}}
\renewcommand{\H}{\mathrm{H}}

\newcommand{\mot}{\mathrm{mot}}

\DeclareFontFamily{U}{min}{}
\DeclareFontShape{U}{min}{m}{n}{<-> udmj30}{}

\makeatletter
\let\c@equation\c@thm

\makeatother

\title{Vector bundles on singular surfaces and threefolds}

\author{Jean Fasel}
\address{Institut Fourier \\
UMR 5582 \\
France} 
\email{\href{mailto:Jean.Fasel@univ-grenoble-alpes.fr}{Jean.Fasel@univ-grenoble-alpes.fr}}
\urladdr{\url{https://www-fourier.ujf-grenoble.fr/~faselj/index.html}}

\author{Maria Yakerson}
\address{CNRS \& IMJ-PRG\\
Paris\\
France}
\email{\href{mailto:yakerson@imj-prg.fr}{yakerson@imj-prg.fr}}
\urladdr{\url{https://www.muramatik.com}}

\date{\today}

\begin{document}

\begin{abstract}
We fully classify vector bundles of critical rank on normal affine surfaces and
threefolds over algebraically closed fields via their Chern classes. We also
prove existence of rank $2$ bundles on threefolds with prescribed first and
second motivic Chern classes. The main input is the motivic filtration on
algebraic $K$-theory: its third piece is generated by smooth closed points and
is identified with motivic cohomology by the third Chern class.
\end{abstract}

\maketitle

\parskip 0.2cm

\parskip 0pt
\tableofcontents

\parskip 0.2cm
\vspace{-2em}

\section{Introduction}

The classification of algebraic vector bundles by characteristic classes is a
classical problem in algebraic geometry. Over an algebraically closed field,
smooth affine surfaces and threefolds provide particularly striking examples.
Murthy and Swan proved that rank $2$ bundles on a smooth affine surface are
classified by their first two Chern classes \cite{MurthySwan76}. On a smooth
affine threefold, Mohan Kumar and Murthy obtained the corresponding
classification in rank $3$ \cite{Mohan82}. Using methods from motivic homotopy
theory, Asok and Fasel established the rank $2$ classification in
characteristic different from $2$ \cite{asok-fasel2}.

We note that the above cases are the only interesting ones for the classification problem. Indeed, line bundles are classified by their first Chern class,
while Serre's splitting theorem \cite{Serre58} says that a vector bundle of
rank greater than the dimension of an affine scheme splits off a trivial line
bundle. Moreover, for vector bundles of rank equal to the dimension of the (affine) variety, the top Chern class is the obstruction
to such a splitting: over a smooth affine variety over an algebraically closed
field, a theorem of Murthy says that a rank $d$ bundle $E$ splits as
$E'\oplus\mathcal O_X$ if and only if $c_d(E)=0$ in $\mathrm{CH}^d(X)$
\cite[Remark~3.6 and Theorem~3.7]{Murthy94}. This low dimensional picture is
special: Mohan Kumar constructed a nontrivial stably trivial rank $2$ bundle on
a smooth complex affine fourfold, and hence a nontrivial bundle with trivial
Chern classes \cite{Kumar85}.

The purpose of this paper is to extend the critical rank classification to
normal affine surfaces and threefolds. For a singular scheme, a central issue
is the choice of cohomological invariants. The motivic cohomology constructed
by Elmanto and Morrow \cite{elmantomorrow} provides a natural framework: it
agrees with classical motivic cohomology on smooth schemes, is related to
algebraic $K$-theory by a multiplicative motivic filtration $F^{\bullet}_{\mot}\mathrm{K}_*$, and satisfies the
projective bundle formula. The latter property gives motivic Chern classes
\[
c_i(E)\in\H^{2i,i}_{\mot}(X,\Z)
\]
for vector bundles on singular schemes. We show that these motivic Chern
classes retain complete information about vector bundles of critical rank,
even in the presence of singularities.

If $\Vect_n(X)$ denotes the set of isomorphism classes of rank $n$ vector
bundles on $X$, our main results are the following.

\begin{thmx}\label{thm:surfaces}
Let $X$ be a normal affine surface over an algebraically closed field. Then the
map induced by motivic Chern classes
\[
(c_1,c_2)\colon \Vect_2(X)\longrightarrow
\H^{2,1}_{\mot}(X,\Z)\times\H^{4,2}_{\mot}(X,\Z)
\]
is a bijection.
\end{thmx}

Since $\H^{2,1}_{\mot}(X,\Z)\simeq\Pic(X)$, the first invariant is simply the
determinant. Thus a rank $2$ bundle on $X$ is uniquely determined by its
determinant and second motivic Chern class, and these invariants may be
prescribed independently.

\begin{thmx}\label{thm:threefolds}
Let $X$ be a normal affine threefold over an algebraically closed field. Then
the map induced by motivic Chern classes
\[
(c_1,c_2,c_3)\colon \Vect_3(X)\longrightarrow
\H^{2,1}_{\mot}(X,\Z)\times\H^{4,2}_{\mot}(X,\Z)
\times\H^{6,3}_{\mot}(X,\Z)
\]
is a bijection.
\end{thmx}

The proofs proceed through the motivic filtration on $\mathrm{K}_0(X)$, drawing
inspiration from the strategy of Mohan Kumar and Murthy \cite{Mohan82} while
developing new arguments for the motivic filtration and the singular setting.
Serre's splitting theorem and Suslin's cancellation theorem \cite{Suslin77}
reduce the classification in rank equal to the dimension to a computation in
$\mathrm{K}_0$. We then compare the successive invariants supplied by the motivic
filtration with Chern classes. The first two are rank and determinant, while
the projection
\[
p_2\colon F^2_{\mot}\mathrm{K}_0(X)\longrightarrow F^2_{\mot}\mathrm{K}_0(X)/F^3_{\mot}\mathrm{K}_0(X)\simeq \H^{4,2}_{\mot}(X,\Z)
\]
is $-c_2$. In dimension two, motivic vanishing completes the argument.

Our main technical result in dimension three is a concrete geometric
description of the final nonzero piece of the motivic filtration. If
$F^3_{\mathrm{sm}}\mathrm{K}_0(X)$ denotes the subgroup generated by smooth closed
points, we prove the following comparison.

\begin{thmx}\label{thm:third-piece-intro}
Let $X$ be a normal affine threefold over an algebraically closed field. Then
\[
F^3_{\mot}\mathrm{K}_0(X)=F^3_{\mathrm{sm}}\mathrm{K}_0(X).
\]
\end{thmx}

Besides its role in the classification theorem, this identifies an abstractly
defined motivic filtration piece with a concrete group generated by smooth
closed points. This group is the image of the Levine--Weibel Chow group of
zero cycles under the cycle class map. It is natural to expect a further
identification with the appropriate Euler class group in the spirit of \cite[Theorem 1.6]{Krishna19} or \cite{Schlichting15}, which would place the
splitting criterion for rank $3$ bundles directly in an Euler class framework.

As a further consequence, we solve the existence problem for rank $2$ bundles
with prescribed first and second motivic Chern classes.

\begin{corx}\label{cor:rank-two-intro}
Let $X$ be a normal affine threefold over an algebraically closed field. Then
the map
\[
(c_1,c_2)\colon \Vect_2(X)\longrightarrow
\H^{2,1}_{\mot}(X,\Z)\times\H^{4,2}_{\mot}(X,\Z)
\]
is surjective.
\end{corx}

Indeed, Theorem~\ref{thm:threefolds} realizes any prescribed pair by a rank
$3$ bundle with vanishing third Chern class, and the splitting criterion yields
a trivial line summand. Whether these two classes also determine a rank $2$
bundle uniquely on a singular affine threefold remains an open question.

Finally, we expect normality to be an artifact of our proofs rather than a
feature of the results.

\subsection*{Structure of the paper}
In Section~\ref{sec: mot coh} we recall motivic cohomology and its associated
filtration on algebraic $K$-theory. Section~\ref{sec:Chern} discusses motivic
Chern classes, and Sections~\ref{sec:properties-filtration} and
\ref{sec:second} identify the low degree pieces of the filtration and prove
the surface classification. Section~\ref{sec:third-piece} proves the key
comparison theorem for the third filtration piece and uses it to establish
the complete threefold classification and the rank $2$ existence theorem.







\subsection*{Notation and conventions}
$k$ is a fixed base field, which we assume to be algebraically closed. We denote by $\Sch_k^{\op}$ the category of qcqs schemes over a field $k$. Affine schemes are always assumed to be of finite type over the base field.

\subsection*{Acknowledgements}
The authors thank Matthew Morrow for answering their questions about the motivic filtration and motivic cohomology.

The authors are grateful to the Institut de Mathématiques de Jussieu --- Paris Rive Gauche and CNRS for the research visit of Jean Fasel to Paris that allowed to pursue this project.

\section{Recollection on the motivic filtration and motivic cohomology}\label{sec: mot coh}

In recent years, progress has been made in defining and analyzing motivic cohomology of singular schemes. The starting point of the construction performed in \cite{elmantomorrow} is the so-called \emph{motivic filtration} $F^\bullet \mathrm{K}$, which can be defined as a functor
\[
F^{\bullet}_{\mot}\mathrm{K}\colon  \Sch_k^{\op}\to \mathrm{CAlg}(\mathrm{FilSp}).
\]
In other words, there exists a functorial filtration on the $K$-theory spectrum which associates to any qcqs scheme $X$ over $k$ a filtered $S^1$-spectrum $F^{\bullet}_{\mot}\mathrm{K}(X)$, together with a multiplication  
\[
F^{i}_{\mot}\mathrm{K}(X)\otimes F^{j}_{\mot}\mathrm{K}(X)\to F^{i+j}_{\mot}\mathrm{K}(X).
\] 
We note that $F^{0}_{\mot}\mathrm{K}(X)=\mathrm{K}(X)$, and we set $F^{j}_{\mot}\mathrm{K}_i(X):=\mathrm{Im}(\pi_iF^{j}_{\mot}\mathrm{K}(X)\to \pi_i\mathrm{K}(X))$ for any $i\in \Z,$ and $j\in\N$. By definition, $F^{j}_{\mot}\mathrm{K}_i(X)\subset \mathrm{K}_i(X)$.

This filtration allows one to build exact triangles of the form 
\begin{equation}\label{eqn:filteredtriangle}
F^{j+1}_{\mot}\mathrm{K}\to F^{j}_{\mot}\mathrm{K}\to \Z(j)_\mot[2j]
\end{equation}
for any $j\geq 0$, where the right-hand term is defined below.

\begin{defn}
The \textit{motivic complex of weight $j$} is the Nisnevich sheaf 
\[
\Z(j)_\mot \colon \Sch_k^{\op} \to \D(\Z),
\] 
defined in \cite[Theorem 1.1(1)]{elmantomorrow}.
The \textit{motivic cohomology group of degree $i$ and weight $j$} of a $k$-scheme $X$  is the group 
\[
\mathrm{H}^{i,j}_\mot(X, \Z) := \mathrm{H}^i(\Z(j)_\mot(X)).
\]
\end{defn}

In particular, for any $k$-scheme $X$, the triangles \eqref{eqn:filteredtriangle} yield a functorial, multiplicative \textit{motivic filtration}  $F^\bullet_{\mot} \mathrm{K}(X)$ on the (non-connective) algebraic $K$-theory $\mathrm{K}(X)$ such that the graded pieces are naturally given by 
    \[\gr^j F^\bullet_{\mot} \mathrm{K}(X) \simeq \Z(j)_\mot(X)[2j]\]
for $j \geq 0$. This gives the motivic spectral sequence
\[ E_2^{ij}(X) = \mathrm{H}^{i-j,-j}_\mot(X, \Z) \Rightarrow \mathrm{K}_{-i-j}(X).\]
Whenever $X$ has finite valuative dimension, the filtration is complete and the spectral sequence is convergent \cite[Theorem 1.1(1)]{elmantomorrow}. Let us now list a few important properties of motivic cohomology that we will use:
\begin{enumerate}
\item(weight $0$) There is a natural weak equivalence $\Z_{\mot}(0)(X)\simeq \mathrm{R}\Gamma_{\mathrm{cdh}}(X,\Z)$, yielding natural isomorphisms 
\[
\H^{i,0}_{\mot}(X,\Z)\simeq \H_{\mathrm{cdh}}^i(X,\Z).
\]
\item(weight $1$) There is a natural map
\[
\mathrm{R}\Gamma_{\mathrm{Nis}}(X,\mathbb{G}_m)[-1]\to \Z_{\mot}(1)
\]
which is an equivalence in degrees $\leq 3$. In particular, for $i\leq 2$
\[
\H^{i-1}_{\mathrm{Nis}}(X,\mathbb{G}_m)\simeq \H^{i,1}_{\mot}(X,\Z)
\]
and thus
\begin{equation}
\mathrm{H}^{i,1}_{\mathrm{mot}}(X,\Z)\simeq\begin{cases} 0 & \text{ if $i\leq 0$.} \\ \mathcal{O}_X(X)^\times & \text{ if $i=1$.} \\ \mathrm{Pic}(X) & \text{ if $i=2$.}\end{cases}
\end{equation}
\item(vanishing) If $X$ is a Noetherian equicharacteristic scheme of finite dimension $d$, $\H^{i,j}_\mot
(X, \Z) = 0$ for all $i > j + \mathrm{dim}(X)$ (\cite[Theorem 1.13]{elmantomorrow}). 
\item (smooth case) If $X$ is smooth over a field $k$, there is for any $j\geq 0$ a natural weak-equivalence
\[
\Z_{\mot}(j)(X)\simeq z^j(X,\bullet)[-2j],
\]
where the right-hand side is Bloch's cycle complex. In particular, there are isomorphisms
\[
\H^{i,j}_{\mot}(X,\Z)\simeq \H_{\mathrm M}^{i,j}(X,\Z)
\]
for any $i,j\in\N$ where the right-hand side is Voevodsky's motivic cohomology. Furthermore, the motivic filtration $F^{\bullet}_{\mot}\mathrm{K}(X)$ coincides with Voevodsky's slice filtration on $\mathrm{K}(X)$.
\item(relation to \'etale cohomology) For any integer $\ell$ prime to the characteristic of $k$, and any qcqs scheme over $k$, there are natural equivalences
\[
\tau^{\leq j}(\Z_{\mot}(j)(X)/\ell)\simeq \tau^{\leq j}\mathrm{R}\Gamma_{\mathrm{et}}(X,\mu_\ell^{\otimes j})
\] 
for any $j\in \N$, yielding isomorphisms $\H_{\mot}^{p,j}(X,\Z/\ell)\simeq \H^p_{\mathrm{et}}(X,\mu_\ell^{\otimes j})$ for $p\leq j$.
\end{enumerate}


%

We will also use the following results.

\begin{lm}\label{lm:rankmap}
Let $X$ be a separated scheme of finite type over an algebraically closed field $k$.
If $X$ is seminormal, then
\[
\H^0_{\mathrm{cdh}}(X,\mathbb{Z}) \simeq \mathbb{Z}^{\pi_0(X)}.
\]
In particular, if $X$ is further connected, then
\[
\H^0_{\mathrm{cdh}}(X,\mathbb{Z}) \simeq \mathbb{Z}.
\]
\end{lm}

\begin{proof}
We use the description of representable cdh-sheaves in terms of seminormalization.
By \cite[Theorem 1.4]{HuberKellyPCII}, for separated finite type $k$-schemes $X,Y$, one has
\[
h_Y^{cdh}(X) \simeq \operatorname{Hom}_k(X^{\mathrm{sn}},Y),
\]
where $X^{\mathrm{sn}}$ denotes the seminormalization of $X$. Since $X$ is supposed to be seminormal, 
\[
h_Y^{cdh}(X) \simeq \operatorname{Hom}_k(X,Y)
\]
 for any separated finite type $k$-scheme $Y$. Now let $A$ be a finite set, and set
\[
Y_A := \coprod_{a \in A} \Spec(k).
\]
Then the constant cdh-sheaf $\underline{A}$ is represented by the cdh-sheafification
of $h_{Y_A}$, so
\[
\Gamma_{\mathrm{cdh}}(X,\underline{A})
= h_{Y_A}^{\mathrm{cdh}}(X)
\simeq \operatorname{Hom}_k(X,Y_A).
\]
But a morphism $X \to Y_A$ is the same as a locally constant map from $X$ to the
discrete set $A$, hence
\[
\operatorname{Hom}_k(X,Y_A) \simeq \operatorname{Map}(\pi_0(X),A).
\]
Thus
\begin{equation}\label{eqn:finitecase}
\Gamma_{\mathrm{cdh}}(X,\underline{A}) \simeq \operatorname{Map}(\pi_0(X),A).
\end{equation}

Now consider the obvious map $ \operatorname{Map}(\pi_0(X),\Z)\to \Gamma_{\mathrm{cdh}}(X,\Z)$, which is injective by construction. To prove that it's also onto, take any section $s\in \Gamma_{\mathrm{cdh}}(X,\Z)$. By definition, there is a cdh cover $X=\cup_{i\in I}U_i$ such that $s$ is equal to some integer $n_i\in \Z$ on $U_i$. By \cite[Proposition 12.28]{mvw}, the covering $\cup_{i\in I}U_i\to X$ admits a refinement of the form $V\to X'\xrightarrow{p} X$, where $p$ is a proper morphism (and a cdh covering) and $V\to X'$ is a Nisnevich covering. The scheme $X'$ being finite type over $k$, it is actually quasi-compact and we can find a finite refinement of $V\to X'$. Altogether, we may suppose that $\cup_{i\in I}U_i\to X$ is finite, and thus the global section $s$ takes only finitely many values. We deduce from \eqref{eqn:finitecase} that it's then in the image of $ \operatorname{Map}(\pi_0(X),\Z)\to \Gamma_{\mathrm{cdh}}(X,\Z)$.
\end{proof}

\begin{prop}\label{prop:torsionfree}
Let $X$ be an affine threefold over an algebraically closed field $k$. Then, the group $\H_{\mot}^{6,3}(X,\Z)$ is torsion free.
\end{prop}

\begin{proof}
For the convenience of the reader, we provide a proof (almost identical to the smooth case) when $\mathrm{char}(k)=0$. The general case will appear in the work of Elmanto-Morrow. Let $p$ be a prime number, and consider the exact sequence
\[
\H^{5,3}_{\mot}(X,\Z/p\Z)\to \H^{6,3}_{\mot}(X,\Z)\xrightarrow{p}\H^{6,3}_{\mot}(X,\Z)
\]
induced by multiplication by $p$. It suffices to show that $\H^{5,3}_{\mot}(X,\Z/p\Z)=0$ to prove that $\H^{6,3}_{\mot}(X,\Z)$ is $p$-torsion free. Now, we know from \cite[Theorem 4.10, 4.]{elmantomorrow} that there is a canonical isomorphism $\H^{5,3}_{\mot}(X,\Z/p\Z)\simeq \H^{5,3}_{\mathrm{cdh}}(X,\Z/p\Z)$, where the right-hand side is the cdh-local motivic cohomology developed in \cite{BEM25}. On the other, the singular extension of the Beilinson–Lichtenbaum equivalence obtained in \cite[Corollary 7.20]{BEM25} provides us with an isomorphism (choosing a primitive $p$-th root of unity)
\[
\Z_{\mathrm{cdh}}(3)/(p)\simeq L_{\mathrm{cdh}}\tau^{\leq 3}R\Gamma_{\mathrm{et}}(-,\Z/p\Z).
\]
By affine Artin vanishing \cite[Chapter 6, Theorem 7.2]{Milne:1980}, $\H_{\mathrm{et}}^j(X,\Z/p\Z)=0$ for $j\geq 4$, and consequently the Zariski sheaf associated to $U\mapsto \H^j_{\mathrm{et}}(X,\Z/p\Z)$ vanishes for $j\geq 4$. The associated cdh-sheaf also vanishes and therefore
\[
L_{\mathrm{cdh}}R\Gamma_{\mathrm{et}}(-,\Z/p\Z)\simeq L_{\mathrm{cdh}}\tau^{\leq 3}R\Gamma_{\mathrm{et}}(-,\Z/p\Z).
\] 
Finally, observe that $R\Gamma_{\mathrm{et}}(-,\Z/p\Z)$ is already cdh-local by \cite[Chapter 59, Lemma 59.105.1]{stacks} and therefore 
\[
\H^{5,3}_{\mot}(X,\Z/p\Z)\simeq \H_{\mathrm{et}}^5(X,\Z/p\Z)=0.
\]
\end{proof}


Finally, we will make use of motivic cohomology groups with support.

\begin{defn}
Let $X$ be a qcqs scheme, and let $Y\subset X$ be a closed subset (viewed as a scheme endowed with the reduced structure). We set $U:=X\smallsetminus Y$ and consider the Nisnevich sheaf $\Z[X/U]\colon \Sch_k^{\op}\to \D(\Z)$ that allows one to obtain fiber sequences
\[
\Z_{\mot}(j)(X/U)\to \Z_{\mot}(j)(X)\to \Z_{\mot}(j)(U)
\]  
and 
\[
F_{\mot}^j\mathrm{K}(X/U)\to F_{\mot}^j\mathrm{K}(X)\to F_{\mot}^j\mathrm{K}(U)
\]
for any $j\geq 0$.
We then set 
\[
 F_{\mot}^j\mathrm{K}_{i,Y}(X):=\mathrm{Im}(\pi_iF_{\mot}^j\mathrm{K}(X/U)\to \pi_i\mathrm{K}(X/U))
\]
and 
\[
\mathrm{H}_{\mot,Y}^{2j-i,j}(X,\Z):=\pi_i(\Z_{\mot}(j)(X/U)).
\]
We will also make use of the following groups:
\[
 F_{\mot}^j\mathrm{K}_{i}(Y;X):=\mathrm{Im}(F_{\mot}^j\mathrm{K}_{i,Y}(X)\to F_{\mot}^j\mathrm{K}_{i}(X)),
\]
and 
\[
\mathrm{H}_{\mot}^{2j-i,j}(Y;X):=\mathrm{Im}(\mathrm{H}_{\mot,Y}^{2j-i,j}(X,\Z)\to \mathrm{H}_{\mot}^{2j-i,j}(X,\Z)).
\]
\end{defn}

If $X$ is a qcqs scheme over a perfect field $k$, we consider $X_{\mathrm{smooth}}\subset X$ the (possibly empty) open subscheme of the smooth locus. If $x\in X_{\mathrm{smooth}}$ is a closed point, we can consider the bicartesian square
\[
\xymatrix{X_{\mathrm{smooth}}\cap (X\smallsetminus\{x\})\ar[r]\ar[d] & X\smallsetminus\{x\}\ar[d] \\
X_{\mathrm{smooth}}\ar[r] & X}
\]
which provides us with a cartesian square
\[
\xymatrix{\mathrm{K}(X)\ar[r]\ar[d] & \mathrm{K}(X\smallsetminus\{x\})\ar[d] \\
\mathrm{K}(X_{\mathrm{smooth}})\ar[r] & \mathrm{K}(X_{\mathrm{smooth}}\cap (X\smallsetminus\{x\}))}
\]
(since $\mathrm{K}$ is Zariski local). In particular, we obtain 
\[
\mathrm{K}_{0}(X/(X\smallsetminus\{x\}))\simeq \mathrm{K}_0(X_{\mathrm{smooth}}/(X_{\mathrm{smooth}}\smallsetminus \{x\}))\simeq \Z
\]
by purity (for $X_{\mathrm{smooth}}$) and then a homomorphism
\begin{equation}\label{eqn:varphix}
\varphi_x\colon \Z\to \mathrm{K}_{0}(\{x\};X)\subset \mathrm{K}_{0}(X).
\end{equation}

\begin{lm}
Suppose that $X$ is of (pure) dimension $d$ and that $x$ is a closed point in the smooth locus. Then $\mathrm{K}_{0}(\{x\};X)\subset F^d_{\mot}\mathrm{K}_0(X)$.
\end{lm}


\begin{proof}
First, we observe that 
\[
\mathrm{H}^{i,j}_{\mot}(X/(X\smallsetminus\{x\}),\Z)\simeq\mathrm{H}^{i,j}_{\mot}(X_{\mathrm{smooth}}/(X_{\mathrm{smooth}}\smallsetminus\{x\}),\Z)\simeq \mathrm{H}^{i-2d,j-d}(k(x),\Z)
\]
by purity, yielding in particular $\mathrm{H}^{i,j}_{\mot}(X/(X\smallsetminus\{x\}),\Z)=0$ if $j-d<0$.  The fiber sequences
\[
F_{\mot}^{i+1}\mathrm{K}\to F_{\mot}^{i}\mathrm{K}\to \Z(i)_{\mot}[2i]
\]
yield exact sequences
\[
\pi_0F_{\mot}^{i+1}\mathrm{K}(X/(X\smallsetminus \{x\}))\to  \pi_0F_{\mot}^{i}\mathrm{K}(X/(X\smallsetminus \{x\}))\to \H_{\mot}^{2i,i}(X/(X\smallsetminus \{x\}),\Z)
\]
with a trivial right-hand group provided $i<d$. We deduce that the composite map 
\[
\pi_0F_{\mot}^{d}\mathrm{K}(X/(X\smallsetminus \{x\}))\to \pi_0F_{\mot}^{d-1}\mathrm{K}(X/(X\smallsetminus \{x\}))\to \ldots\to \mathrm{K}_0(X/(X\smallsetminus \{x\}))
\]
is onto and we conclude using the commutative diagram
\[
\xymatrix{\pi_0F_{\mot}^{d}\mathrm{K}(X/(X\smallsetminus \{x\}))\ar[r] \ar[d]& \mathrm{K}_0(X/(X\smallsetminus \{x\}))\ar[d]^-{\varphi_x}  \\
\pi_0F_{\mot}^{d}\mathrm{K}(X)\ar[r] & \mathrm{K}_0(X). }
\]
\end{proof}

\section{Motivic Chern classes}\label{sec:Chern}

Motivic cohomology satisfies the projective bundle formula, as proved in \cite[Theorem 1.1(7)]{elmantomorrow}. Consequently, there are constructed motivic Chern classes with values in motivic cohomology:
\[c_i \colon \Vect_n(X) \to \H^{2i,i}_\mot (X, \Z); \, \, i \in \N.  \]



A formal direct consequence of the projective bundle theorem is the so-called \emph{splitting principle}.

\begin{thm}[Splitting principle]
    For any qcqs scheme $X$ and any vector bundle $V$ over $X$, there exists a qcqs scheme $q\colon Y\to X$ over $X$ such that $q^*V$ is obtained as successive extensions of line bundles and the homomorphism 
\[
q^*\colon \mathrm{H}^{i,j}_{\mathrm{mot}}(X,\Z)\to \mathrm{H}^{i,j}_{\mathrm{mot}}(Y,\Z)
\]
is injective for any $i,j\in\Z$. 
\end{thm}

Thus, when computing the Chern classes of a vector bundle $V$, we may always suppose that $V$ is filtered by line bundles and that its class in $\mathrm{K}_0(X)$ is a sum of line bundles.  

The following proposition follows from the construction of Chern classes.

%

\begin{prop}
    The first Chern class factors as 
\[
\mathrm{K}_0(X)\xrightarrow{\det }\mathrm{Pic}(X)\to \H_{\mot}^{2,1}(X,\Z),
\]
where the right-hand map is the isomorphism induced by the equivalence $\mathrm{R}\Gamma_{\mathrm{Nis}}(X,\mathbb{G}_m)[-1]\to \Z_{\mot}(1)$.
\end{prop}

\begin{proof}
This is true by definition if $V$ is a line bundle. Else, we may use the splitting principle to assume that $V$ is a sum of line bundles $L_1,\ldots,L_n$. The determinant is then the product $L_1\otimes\cdots\otimes L_n$ and we conclude from the fact that the first Chern class is additive. 
\end{proof}



\section{Some properties of the motivic filtration}\label{sec:properties-filtration}


By definition, the motivic filtration on $K$-theory is given by the image of the filtered pieces. More precisely, we have surjective homomorphisms $\pi_j(F_{\mot}^i\mathrm{K}(X))\to F^i_{\mathrm{mot}}\mathrm{K}_j(X)$ for any $i,j\in\Z$.

\begin{lm}\label{lem:filtration}
For $0\leq i\leq 2$ and $j\geq 0$, the homomorphism $\pi_j(F_{\mot}^i\mathrm{K}(X))\to F^i_{\mathrm{mot}}\mathrm{K}_j(X)$ is an isomorphism in case $X$ is normal.
\end{lm}

\begin{proof}
If $i=0$, there is nothing to prove. For $i=1$, it suffices to show that the maps
\[
\pi_j(F^1_{\mot}\mathrm{K}(X))\to F^1_{\mathrm{mot}}\K_j(X)
\]
are injective. The kernel of this map is controlled by $\pi_{j+1}\Z_{\mathrm{mot}}(0)[0](X)=\mathrm{H}_\mathrm{mot}^{-1-j,0}(X,\Z)=0$ if $j\geq 0$. 

For $i=2$, we have  an exact sequence
\[
\mathrm{H}_\mathrm{mot}^{2-(1+j),1}(X,\Z)\to \pi_j(F_{\mot}^2\K(X))\to \pi_j(F_{\mot}^1\K(X))
\]
and the left hand term vanishes in case $j\geq 1$. We conclude in that case using the fact that $\pi_j(F^1\K(X))\to \K_j(X)$ is injective. We are thus left with $j=0$, in which case we have a diagram of the form
\[
\xymatrix{\pi_1(F_{\mot}^1\K(X))\ar[r]\ar[d] & \mathrm{H}_\mathrm{mot}^{1,1}(X,\Z)\ar[r] & \pi_0(F_{\mot}^2K(X))\ar[r] & \pi_0(F_{\mot}^1\K(X)) \\
\K_1(X) & & & }
\]
in which the vertical arrow is an isomorphism since $\mathrm{H}_\mathrm{mot}^{i,0}(X,\Z)=0$ for $i<0$. We are thus left to show that the map 
\[
\K_1(X)\to \mathrm{H}_\mathrm{mot}^{1,1}(X,\Z)\simeq\mathcal{O}_X(X)^\times
\]
is onto, which follows from the fact that it is the determinant map. Indeed, this is true if $X$ is smooth. Now, $X$ being normal implies that the restriction map $\mathcal{O}_X(X)\to \mathcal{O}_U(U)$ is an isomorphism, and consequently $\mathcal{O}_X(X)^\times=\mathcal{O}_U(U)^\times$. We conclude using the commutative diagram  
\[
\xymatrix{\K_1(X)\ar[r]\ar[d] & \mathrm{H}_\mathrm{mot}^{1,1}(X,\Z)\ar[r]^-{\sim}\ar[d] & \mathcal{O}_X(X)^\times\ar@{=}[d] \\ 
\K_1(U)\ar[r] & \mathrm{H}_\mathrm{mot}^{1,1}(U,\Z)\ar[r]^-{\sim} & \mathcal{O}_U(U)^\times \\ }
\]
\end{proof}

We will use the following lemma later.

\begin{lm}\label{lem:exhaustive}
If $X$ is a noetherian scheme of dimension $d$ over a field $k$ and $i\geq d$, $\pi_j(F_{\mot}^i\K(X))=0$ for $j<0$. If further $i>d$, then $\pi_0(F_{\mot}^i\K(X))=0$ as well.
\end{lm}

\begin{proof}
We start with the second assertion. We have an exact sequence 
\[
\mathrm{H}^{2d+1,d+1}_{\mot}(X,\Z)\to \pi_0(F_{\mot}^{d+2}\K(X))\to \pi_0(F_{\mot}^{d+1}\K(X))\to \mathrm{H}^{2d+2,d+1}_{\mot}(X,\Z)
\]
and the right-hand group is trivial by vanishing. Repeating the process, we see that the composite map 
\[
\pi_0(F_{\mot}^{d+l}\K(X))\to \ldots \to \pi_0(F_{\mot}^{d+2}\K(X)) \to \pi_0(F_{\mot}^{d+1}\K(X))
\]
is onto for any $l\geq 2$. It suffices then to show that $\pi_0(F_{\mot}^{d+l}K(X))=0$ for $l$ large enough. This follows from the fact that the filtration is bounded \cite[Theorem 4.12, Theorem 4.27]{elmantomorrow}.

We now prove the first assertion, starting with the exact sequence
\[
\mathrm{H}^{2d-j+1,d}_{\mot}(X,\Z)\to \pi_j(F_{\mot}^{d+1}\K(X))\to \pi_j(F_{\mot}^{d}\K(X))\to \mathrm{H}^{2d-j,d}_{\mot}(X,\Z)
\]
Since $j<0$, we see that $\mathrm{H}^{2d-j,d}_{\mot}(X,\Z)=0$ by vanishing. Arguing in the same way, we find that all morphisms in the sequence
\[
\pi_j(F_{\mot}^{d+l}\K(X))\to\cdot\to \pi_j(F_{\mot}^{d+1}\K(X))\to \pi_j(F_{\mot}^{d}\K(X))
\]
are onto. We conclude as in the previous case using the fact that the filtration is bounded.
\end{proof}

As a consequence of these results, we can compare the multiplication maps
\[
F_{\mot}^i\K(X)\otimes F_{\mot}^j\K(X)\to F_{\mot}^{i+j}\K(X)
\]
induced by the motivic filtration with the ones in motivic cohomology.
If $X$ is any quasi-compact quasi-separated scheme over a field $k$, the map of graded commutative algebras $p_*\colon F_{\mot}^{\bullet}\K\to \Z(\bullet)[2\bullet]$ yields a commutative diagram
\[
\xymatrix{F_{\mot}^i\K(X)\otimes F_{\mot}^j\K(X)\ar[r]^-{p_i\otimes p_j}\ar[d] & \Z_X(i)[2i]\otimes \Z_X(j)[2j]\ar[d] \\
F_{\mot}^{i+j}\K(X)\ar[r]_-{p_{i+j}} & \Z_X(i+j)[2i+2j] }
\]
where the vertical maps are the multiplication maps. Taking $\pi_0(-)$ and using Lemma \ref{lem:filtration}, we obtain a commutative diagram
\begin{equation}\label{eqn:pmultiplicative}
\xymatrix{F_{\mot}^i\K_0(X)\otimes F_{\mot}^j\K_0(X)\ar[r]^-{p_i\otimes p_j}\ar[r]^-{p_i\otimes p_j}\ar[d] & \mathrm{H}_{\mathrm{mot}}^{2i,i}(X,\Z)\otimes \mathrm{H}_{\mathrm{mot}}^{2j,j}(X,\Z)\ar[d] \\
\pi_0(F_{\mot}^{i+j}\K(X))\ar[r]_-{p_{i+j}} & \mathrm{H}_{\mathrm{mot}}^{2i+2j,i+j}(X,\Z) }
\end{equation}
in case $i,j\leq 2$.

\section{First and second pieces of the filtration}\label{sec:second}

In this section, we start with the identification of the maps $p_i\colon \pi_0(F_{\mot}^i\K(X))\to \H_{\mot}^{2i,i}(X,\Z)$ for $i\leq 2$. In view of Lemma \ref{lem:filtration}, this amounts to understand the associated maps 
\[
p_i\colon F_{\mot}^i\K_0(X)\to \H_{\mot}^{2i,i}(X,\Z).
\]

\begin{lm}\label{lm:rank}
If $X$ is seminormal, the map $p_0\colon \K_0(X)\to  \H_{\mot}^{0,0}(X,\Z)$ coincides with the rank map under the isomorphism of Lemma \ref{lm:rankmap}.
\end{lm}

\begin{proof}
We may assume that $X$ is connected, and reduced. The smooth locus $U$ is then dense and we conclude from the smooth case using the fact that $ \H_{\mot}^{0,0}(X,\Z)\to \H_{\mot}^{0,0}(U,\Z)$ is injective.
\end{proof}

\begin{lm}\label{lm:determinant}
If $X$ is normal, the map $p_1\colon F_{\mot}^1\K_0(X)\to \H_{\mot}^{2,1}(X,\Z)$ coincides with the first Chern class $c_1$.
\end{lm}

\begin{proof}
Let $U\subset X$ be the smooth locus. Since $X$ is normal, the restriction map $\mathrm{Pic}(X)\to \mathrm{Pic}(U)$ is injective and we can conclude using the commutative diagram 
\[
\xymatrix{F^1_{\mot}\K_0(X)\ar[r]^-{p_1-c_1}\ar[d] & \H_{\mot}^{2,1}(X,\Z)\simeq\mathrm{Pic}(X)\ar[d] \\
F^1_{\mot}\K_0(U)\ar[r]^-{p_1-c_1}\ar[r] & \H_{\mot}^{2,1}(U,\Z)\simeq\mathrm{Pic}(U)}
\] 
and the fact that the claim is true when $X$ is smooth, as the motivic filtration coincides in that case with Voevodsky's slice filtration.
\end{proof}

\begin{lm}\label{lm:secondchernclass}
If $X$ is normal, the second Chern class
\[
c_2\colon F^2_{\mot}\K_0(X)\to \mathrm{H}_{\mot}^{4,2}(X,\Z)
\]
is a group homomorphism.
\end{lm}

\begin{proof}
If $\alpha\in F^2_{\mot}\K_0(X)\subset \K_0(X)$, we may write $\alpha=[V_1]-[W_1]$ for some vector bundles $V_1,W_1$ over $X$. The composites
\[
F^2_{\mot}\K_0(X)\to F^0_{\mot}\K_0(X)\to \H_{\mot}^{0,0}(X,\Z)
\]
and 
\[
F^2_{\mot}\K_0(X)\to F^1_{\mot}\K_0(X)\to \H_{\mot}^{2,1}(X,\Z)
\]
are trivial by definition, and we deduce from the fact that the first map factors through the rank map that $V_1$ and $W_1$ are of the same rank. We also deduce from the second one that they have the same determinant (and the same first Chern class). Write now 
\[
c_t(V_1)=1+c_1(V_1)t+c_2(V_1)t^2+\ldots
\]
for the Chern power series of $V_1$ (which is in fact a polynomial as $X$ is finite dimensional), and the same for $W_1$. Since the Whitney sum formula holds, $c_t(-W_1)=c_t(W_1)^{-1}$ as a power series and we then have 
\[
c_t(-W_1)=1-c_1(W_1)t+(c_1(W_1)^2-c_2(W_1))t^2+\ldots=1-c_1(V_1)t+(c_1(V_1)^2-c_2(W_1))t^2+\ldots
\]
yielding in turn
\[
c_t(\alpha)=c_t(V_1)c_t(-W_1)=1+(c_2(V_1)-c_2(W_1))t^2+\ldots
\]
and so $c_2(\alpha)=c_2(V_1)-c_2(W_1)$. If $\beta=[V_2]-[W_2]$, the same formula holds and then 
\[
c_2(\alpha+\beta)=c_2([V_1\oplus V_2]-[W_1\oplus W_2])=c_2([V_1\oplus V_2])-c_2([W_1\oplus W_2]).
\]
As $c_2([V_1\oplus V_2])=c_2(V_1)+c_1(V_1)c_1(V_2)+c_2(V_2)$ and $c_2([W_1\oplus W_2])=c_2(W_1)+c_1(W_1)c_1(W_2)+c_2(W_2)$, we conclude using $c_1(V_1)=c_1(W_1)$ and $c_1(V_2)=c_1(W_2)$.
\end{proof}

\begin{rmk}
Note that the Chern class $c_2\colon F^1_{\mot}\K_0(X)\to \mathrm{H}_{\mot}^{4,2}(X,\Z)$ is in general not a group homomorphism.
\end{rmk}

\begin{prop}\label{prop:secondchernclass}
For any normal scheme $X$, the homomorphism 
\[
p_2\colon F^2_{\mathrm{mot}}\K_0(X)\to \mathrm{H}_{\mathrm{mot}}^{4,2}(X,\Z)
\]
is equal to $-c_2$.
\end{prop}

\begin{proof}
In this proof, we write $1$ for the class $[\mathcal{O}_X]\in K_0(X)$. Let $\alpha=[V]-[W]\in F^2_{\mathrm{mot}}\K_0(X)$. As observed above, the bundles have the same rank and determinant. If $n=\mathrm{rank}(V)\geq 1$, we write 
\[
[V]-[W]=[V]-[\det(V)]-(n-1)+[\det(V)]+(n-1)-[W]
\]
and we are reduced (in view of the above lemma) to prove the result for elements of the form $[V]-[\det(V)]-(n-1)\in F^2_{\mathrm{mot}}K_0(X)$. If $V$ is of rank one, there is nothing to prove and we first suppose that $V$ is of rank $2$. Using the splitting principle, we may suppose that $V=L_1\oplus L_2$ for line bundles $L_1$ and $L_2$ over $X$. We then have
\[
([L_1]-1)\cdot ([L_2]-1)=[L_1\otimes L_2]-[L_1]-[L_2]+1=[\det V]+1-[V].
\]
Using the commutative diagram \eqref{eqn:pmultiplicative}, we obtain
\[
p_2([\det V]+1-[V])=p_1([L_1]-1)\cdot p_1([L_2]-1)=c_1(L_1)c_1(L_2).
\]
Now $c_1(L_1)c_1(L_2)=c_2(V)=c_2([V]-[\det V]-[\mathcal{O}_X])$ by the above lemma and the result is proved in that case. 

Suppose that we have proved that 
\[
p_2([V]-[\det(V)]-(n-1))=-c_2([V]-[\det(V)]-(n-1))
\]
for any vector bundle of rank $n\geq 2$ and consider a vector bundle $W$ of rank $n+1\geq 3$. Using the splitting principle, we may suppose that $W=V\oplus L$ for some vector bundle $V$ of rank $n$. We obtain
\begin{eqnarray*}
[W]-[\det(W)]-n & = & [V]-[\det(V)]-(n-1)-1+[L]+[\det(V)]-[\det(W)] \\
 & = & [V]-[\det(V)]-(n-1)-([L]-1)([\det(V)]-1).
\end{eqnarray*}
Thus finally
\begin{eqnarray*}
p_2([W]-[\det(W)]-n) & = & p_2([V]-[\det(V)]-(n-1))-p_2(([L]-1)([\det(V)]-1)) \\
& = & -c_2(V)-c_1(L)c_1(\det(V)) \\
& = & -c_2(V)-c_1(L)c_1(V) \\
& = & -c_2(W).
\end{eqnarray*}
\end{proof}

As a result, we can now obtain classification of vector bundles over surfaces, stated in the introduction.

\begin{proof}[Proof of Theorem \ref{thm:surfaces}]
In view of \cite{Suslin77}, two vector bundles of rank $2$ over $X$ are isomorphic if and only if they have the same class in $\K_0(X)$. Let then $V$ and $W$ be rank $2$ bundles having the same first and second Chern classes. Since they have the same rank, the class $[V]-[W]\in \K_0(X)$ is in fact in $F^1_{\mot}\K_0(X)$ by Lemma \ref{lm:rank}. Further, Lemma \ref{lm:determinant} shows that it is indeed in $F^2_{\mot}\K_0(X)$ and we conclude using Lemma \ref{lm:secondchernclass} and Proposition \ref{prop:secondchernclass} 
that it is actually in $F^3_{\mot}\K_0(X)$ which vanishes in view of Lemma \ref{lem:exhaustive}.

We still need to show that the map 
\[
(c_1, c_2) \colon \Vect_2(X) \to \H^{2,1}_\mot (X, \Z) \times \H^{4,2}_\mot (X, \Z)
\]
is also onto. Let thus $(a_1,a_2)\in \H^{2,1}_\mot (X, \Z) \times \H^{4,2}_\mot (X, \Z)$. Since $\H^{2,1}_\mot (X, \Z)$ is isomorphic to $\mathrm{Pic}(X)$, there exists a line bundle $L$ over $X$ such that $c_1(L)=a_1$. We now prove that the map $p_2\colon F^2_{\mot}\K_0(X)\to \H^{4,2}_\mot (X, \Z)$ is onto. The long exact sequence reads as 
\[
\pi_0(F^2_{\mot}\K(X))\to \H^{4,2}_\mot (X, \Z)\to \pi_{-1}(F^3_{\mot}\K(X))
\]
and we conclude from Lemma \ref{lem:exhaustive} that $\pi_{-1}(F^3_{\mot}\K(X))=0$. By Proposition \ref{prop:secondchernclass}, there exists an element $\alpha\in F^2_{\mot}\K_0(X)$ such that $c_2(\alpha)=a_2$. Using Serre's splitting theorem \cite[\S 4, Theorem 1]{Serre58}, $\alpha=[V]-[\mathcal{O}_X^2]$ with $V$ of rank $2$ and with trivial determinant because $\alpha\in F^2_{\mot}\K_0(X)$. We see that $[V]-[\mathcal{O}_X^2]+[L]-[\mathcal{O}_X]\in \K_0(X)$ has the required Chern classes. We now use again Serre splitting theorem to write $[V\oplus L]$ as $[W\oplus\mathcal{O}_X]$ and conclude that $[W]$ satisfies $c_1(W)=a_1$ and $c_2(W)=a_2$.
\end{proof}

\section{Third piece of the filtration}\label{sec:third-piece}

We now pass to the identification of the third piece in the filtration of $\K_0(X)$, in the case where $X$ is a normal threefold. First, we observe that the map
\[
 \pi_0(F_{\mot}^3K(X))\xrightarrow{p_3}\mathrm{H}^{6,3}_{\mot}(X,\Z)
\]
is an isomorphism, since $\pi_0(F_{\mot}^4K(X))=0=\pi_{-1}(F_{\mot}^4K(X))$ by Lemma \ref{lem:exhaustive}. We thus obtain an \emph{edge homomorphism}
\[
\epsilon_3\colon \mathrm{H}^{6,3}_{\mot}(X,\Z)\to F^3_{\mot}\K_0(X)
\]
defined as the composite $\mathrm{H}^{6,3}_{\mot}(X,\Z)\xrightarrow{p_3^{-1}} \pi_0(F_{\mot}^3K(X))\to F^3_{\mot}\K_0(X)\subset \K_0(X)$. Note that $\epsilon_3$ is onto by definition.

Recall Proposition~\ref{prop:torsionfree}, which claimed that for an affine threefold $X$ over an algebraically closed field, the group $\H_{\mot}^{6,3}(X,\Z)$ is torsion free. We will need the following enhanced version of this result for our proof.

\begin{prop}\label{cor:XminusTtf}
Let $X$ be an affine threefold over an algebraically closed field $k$, and let $T\subset X$ be an integral curve in the smooth locus $X_{\mathrm{smooth}}\subset X$. Then the group $\H_{\mot}^{6,3}(X\smallsetminus T,\Z)$ is torsion free.
\end{prop}

\begin{proof}
We use the exact sequence
\[
\mathrm{H}^{6,3}_{\mot}(X/(X\smallsetminus T),\Z)\to \mathrm{H}^{6,3}_{\mot}(X,\Z)\to \mathrm{H}^{6,3}_{\mot}(X\smallsetminus T,\Z)\to \mathrm{H}^{7,3}_{\mot}(X/(X\smallsetminus T),\Z)
\]
We first prove that $\mathrm{H}^{7,3}_{\mot}(X/(X\smallsetminus T),\Z)=0$. Since $X/(X\smallsetminus T)\simeq X_{\mathrm{smooth}}/(X_{\mathrm{smooth}}\smallsetminus T)$ by Zariski excision, we obtain 
\[
 \mathrm{H}^{7,3}_{\mot}(X/(X\smallsetminus T),\Z) \simeq \mathrm{H}^{7,3}_{\mot}(X_{\mathrm{smooth}}/(X_{\mathrm{smooth}}\smallsetminus T),\Z)
\] 
and we are consequently reduced to show that $\mathrm{H}^{6,3}_{\mot}(X_{\mathrm{smooth}},\Z)\to \mathrm{H}^{6,3}_{\mot}(X_{\mathrm{smooth}}\smallsetminus T,\Z)$ is onto and that $\mathrm{H}^{7,3}_{\mot}(X_{\mathrm{smooth}},\Z)=0$. This is obvious since for a smooth scheme $Y$ of dimension $3$, $H^{6,3}_{\mot}(Y,\Z)\simeq \mathrm{CH}^3(Y)$ and $H^{7,3}_{\mot}(Y,\Z)=0$.

We now prove that 
\[
\mathrm{H}^{6,3}_{\mot}(X/(X\smallsetminus T),\Z) \simeq \mathrm{H}^{6,3}_{\mot}(X_{\mathrm{smooth}}/(X_{\mathrm{smooth}}\smallsetminus T),\Z)
\]
is a divisible group. Once again, we use the fact that 
\[
\mathrm{H}^{6,3}_{\mot}(X_{\mathrm{smooth}}/(X_{\mathrm{smooth}}\smallsetminus T),\Z) \simeq \mathrm{CH}^3_T(X_{\mathrm{smooth}})
\]
which can be computed using a complex of the form
\[
k(T)^\times\to \bigoplus_{x\in X^{(3)}\cap T}\Z.
\]
If we denote by $\tilde T$ the normalization of $T$, we have that $\tilde T$ is affine (since $T$ is), and further that $\mathrm{CH}^1(\tilde T)\to \mathrm{CH}^3_T(X_{\mathrm{smooth}})$ is onto (since the fiber over any point $t\in T$ under the finite morphism $\tilde T\to T$ is a finite number of points with the same residue field, using that $k$ is algebraically closed). The Picard group of a smooth affine curve over an algebraically closed field is divisible (\cite[Example 1.6.6, p.17]{Fulton}), and thus $\mathrm{H}^{6,3}_{\mot}(X/(X\smallsetminus T),\Z)$ is divisible. Its image in $\mathrm{H}^{6,3}_{\mot}(X,\Z)$ is then also divisible.

The result follows now from the observation that if we have an exact sequence of abelian groups 
\[
0\to D\xrightarrow{i} G\xrightarrow{\pi} A\to 0
\] 
with $D$ divisible and $G$ torsion free, then $A$ is also torsion free. Suppose that $a\in A$ is such that $na=0$ for some $n\in\N\smallsetminus\{0\}$. If $a'\in G$ is such that $\pi(a')=a$, there exists $d\in D$ such that $na'=i(d)=i(n\frac dn)=ni(\frac dn)$. Thus $a'-i(\frac dn)$ is torsion, and then trivial. It follows that $a=\pi(a')=\pi(a'-i(\frac dn))=0$.
\end{proof}

\begin{prop}\label{prop:edgeisomorphism}
If $X$ is a normal affine threefold over an algebraically closed field $k$, the edge map 
\[
\epsilon_3\colon \H_{\mot}^{6,3}(X,\Z)\to F^3_{\mot}\K_0(X) 
\]
is an isomorphism. Further, the map $\pi_0(F^3\K(X))\to F^3_{\mot}\K_0(X)\subset \pi_0(\K(X))=\K_0(X)$ is injective.
\end{prop}

\begin{proof}
The edge map is onto by definition, and we are thus left to show that it's also injective. We have a commutative diagram
\[
\xymatrix{\H_{\mot}^{6,3}(X,\Z)\ar[r]^-{\epsilon_3}\ar[d] & F^3_{\mot}\K_0(X)\ar[d] \\
\H_{\mot}^{6,3}(X,\Z)\otimes_{\Z}\mathbb{Q}\ar[r]_-{\epsilon_3\otimes 1} & F^3_{\mot}\K_0(X)\otimes_{\Z}\mathbb{Q}}
\]
in which the vertical maps are localizations. Since $\H_{\mot}^{6,3}(X,\Z)$ is torsion free by Proposition \ref{prop:torsionfree}, the left vertical map is injective and the fact that the Atiyah-Hirzebruch spectral sequence degenerates rationally yields an isomorphism 
\[
\epsilon_3\otimes 1\colon \H_{\mot}^{6,3}(X,\Z)\otimes_{\Z}\mathbb{Q}\to F^3_{\mot}\K_0(X)\otimes_{\Z}\mathbb{Q}.
\]
We deduce that $\epsilon_3$ is indeed injective.  The second statement follows from the observation that the map $\pi_0(F^3\K(X))\to F^3_{\mot}\K_0(X)\subset \pi_0(\K(X))=\K_0(X)$ is the composite of $p_3$ with $\epsilon_3$.

\end{proof}

\begin{rmk}\label{rmk:edgeisomorphism}
Under the hypotheses of Proposition \ref{cor:XminusTtf}, we see that the edge map
\[
\epsilon_3\colon \H_{\mot}^{6,3}(X\smallsetminus T,\Z)\to F^3_{\mot}\K_0(X\smallsetminus T) 
\]
is an isomorphism as well, and that $\pi_0(F^3\K(X\smallsetminus T))\to F^3_{\mot}\K_0(X\smallsetminus T)\subset \pi_0(\K(X\smallsetminus T))=\K_0(X\smallsetminus T)$ is injective.
\end{rmk}

\begin{lm}
Let $X$ be a normal scheme. Then the third Chern class
\[
c_3\colon F^3_{\mathrm{mot}}\K_0(X)\to \mathrm{H}_{\mot}^{6,3}(X,\Z)
\]
is a group homomorphism.
\end{lm}

\begin{proof}
The proof is formally the same as in Lemma \ref{lm:secondchernclass} above, using the crucial fact that $\alpha\in F^3_{\mathrm{mot}}\K_0(X)$ implies that $c_1(\alpha)=c_2(\alpha)=0$ from Lemma \ref{lm:determinant} and Proposition \ref{prop:secondchernclass}.
\end{proof}


Our goal is now to prove that $F^3_{\mathrm{mot}}\K_0(X)$ is supported on smooth closed points in our situation. We will make use of the homomorphisms $\varphi_x$ defined in \eqref{eqn:varphix}.

\begin{defn}
Let $X$ be a qcqs scheme over a perfect field $k$ of (pure) dimension $d$. We denote by $F^d_{\mathrm{sm}}\K_0(X)\subset F^d_{\mot}\K_0(X)$ the subgroup generated by smooth closed points, i.e. the image of 
\[
\bigoplus_{x\in X_{\mathrm{closed,\phantom{i} smooth}}} \Z\xrightarrow{\oplus\varphi_x} F^d_{\mot}\K_0(X)\subset \K_0(X).
\]
\end{defn}

\begin{rmk}
The group $F^d_{\mathrm{sm}}\K_0(X)$ is usually simply denoted by $F^d\K_0(X)$. In this paper, we use the former notation to avoid confusion with the group $F^d_{\mot}\K_0(X)$.
\end{rmk}

\begin{thm}\label{thm:third-piece-smooth-points}
If $X$ is a normal affine threefold over an algebraically closed field $k$, 
\[
F^3_{\mathrm{sm}}\K_0(X)=F^3_{\mot}\K_0(X).
\]
\end{thm}


\begin{proof}
First, we observe that any element in $F^3_{\mot}\K_0(X)$ is of the form $[P]-[\mathcal{O}_X^3]$ with $P$ of rank $3$ such that $c_1(P)=c_2(P)=0$ (use Lemma \ref{lm:determinant} and Proposition \ref{prop:secondchernclass}). By Swan's version of Bertini theorem (\cite[\S 3]{Mohan82}), there exists a section $s\colon P\to \mathcal{O}_X$ whose zero locus is several smooth closed points, say $\{x_1,\ldots,x_n\}\subset X$. The kernel $Q$ of $s$ is finitely generated, and there exists a surjection $\mathcal{O}_X^m\to Q$ for some $m\in\N$. If $U=X\smallsetminus\{x_1,\ldots,x_n\}$, then we have an exact sequence
\[
0\to Q_{\vert U}\to P_{\vert U}\to \mathcal{O}_U\to 0
\]
and in particular $Q_{\vert U}$ is locally free. Moreover, we have a surjective map $\mathcal{O}_U^m\to Q_{\vert U}$ and $Q_{\vert U}$ is still generated by global sections. By Bertini once again, there exists a section $t\colon Q_{\vert U}\to \mathcal{O}_U$ such that the vanishing locus $Y$ of $t$ satisfies the conditions:
\begin{enumerate}
\item $Y$ is of pure codimension $2$.
\item $Y$ is smooth outside the singular locus of $U$.
\item $Y$ is irreducible.
\item $Y$ is disjoint from any closed subset $Z\subset U$ of codimension $\geq 2$.
\end{enumerate}
If $Z_1$ is the singular locus of $X$, then it is also the singular locus of $U$ and is of codimension $\geq 2$ by assumption. We may then assume that $Y\subset U$ is smooth.
If $\overline Y$ is the closure of $Y$ in $X$, we observe that $\overline Y$ might be singular (at points in $(\overline Y\smallsetminus Y)\subset\{x_1,\ldots,x_n\}$) and we set $T:=\overline Y\cup \{x_1,\ldots,x_n\}$; $V=X\smallsetminus T$. On $V$, we have an exact sequence
\[
0\to Q_{\vert V}\to P_{\vert V}\to \mathcal{O}_V\to 0
\]
and a surjection $Q_{\vert V}\to \mathcal{O}_V$. Now, $P$ is of trivial determinant and thus $Q_{\vert V}$ is also of trivial determinant (and of rank $2$). Consequently, we obtain an exact sequence 
\[
0\to \mathcal{O}_V\to Q_{\vert V}\xrightarrow{t_{\vert V}} \mathcal{O}_V\to 0
\]
and we finally see that the class $[P]-[\mathcal{O}_X^3]$ vanishes in $\K_0(V)$. In other words, it is in $\K_{0,T}(X)$, with $T$ in the smooth locus of $X$. We have a cocartesian square 
\[
\xymatrix{X\smallsetminus T\ar[r]\ar[d] & X\smallsetminus \overline Y\ar[d] \\
X\smallsetminus \{y_1,\ldots,y_r\}\ar[r] & X}
\]
where $\{y_1,\ldots,y_r\}$ is the (possibly empty) set of points in $\{x_1,\ldots,x_n\}$ that are not in $\overline Y$,
yielding a commutative diagram of the form
\[
\xymatrix{ 
\pi_0\K(X/(X\smallsetminus \{y_1,\ldots,y_r\}))\ar[r]\ar@{=}[d] & \K_0(X)\ar[r]\ar[d] &  \K_0(X\smallsetminus \{y_1,\ldots,y_r\})\ar[d] \\
\pi_0\K(X/(X\smallsetminus \{y_1,\ldots,y_r\}))\ar[r] & \K_0(X\smallsetminus \overline Y)\ar[r] &  \K_0(X\smallsetminus T).}
\]
By purity (the closed points are in the smooth locus), we have 
\[
\pi_0\K(X/(X\smallsetminus \{y_1,\ldots,y_r\}))=\pi_0F_{\mot}^3\K(X/(X\smallsetminus \{y_1,\ldots,y_r\})) \simeq \bigoplus \Z
\]
and the diagram becomes 
\[
\xymatrix{ 
\bigoplus \Z\ar[r]^-f\ar@{=}[d] & \K_0(X)\ar[r]\ar[d]_-i &  \K_0(X\smallsetminus \{y_1,\ldots,y_r\})\ar[d] \\
\bigoplus \Z\ar[r]_-g & \K_0(X\smallsetminus \overline Y)\ar[r]_-j &  \K_0(X\smallsetminus T)}
\]
Letting $\alpha:=[P]-[\mathcal{O}_X^3]\in F^3_{\mot}\K_0(X)\subset \K_0(X)$, we see that its image under the composite 
\[
\K_0(X)\to \K_0(X\smallsetminus \overline Y)\to  \K_0(X\smallsetminus T)
\]
is trivial, i.e. $ji(\alpha)=0$. There exists thus $a\in\bigoplus \Z$ such that $i(\alpha)=g(a)$, and we obtain that $i(\alpha-f(a))=0$. Note that the image of $a$ is in $F^3_{\mathrm{sm}}\K_0(X)$ by construction. Replacing $\alpha$ by $\alpha-f(a)$, we may then suppose that $i(\alpha)=0$. Note further that $i$ induces a homomorphism $i\colon F^3_{\mot}\K_0(X)\to F^3_{\mot}\K_0(X\smallsetminus \overline{Y})$ and finally we can suppose that 
\[
\alpha\in \ker (F^3_{\mot}\K_0(X)\to F^3_{\mot}\K_0(X\smallsetminus \overline{Y})).
\]
Using Proposition \ref{prop:edgeisomorphism} and Remark \ref{rmk:edgeisomorphism}, we see that $\alpha$ is in fact in the image of the group $\pi_0F_{\mot}^3\K(X/(X\smallsetminus \overline{Y}))$ under the map
\[
\pi_0F^3_{\mot}\K(X/(X\smallsetminus \overline{Y}))\to \pi_0F^3_{\mot}\K(X)\simeq F^3_{\mot}\K_0(X).
\]
Now $\overline Y$ is an integral curve in the smooth locus of $X$, and consequently 
\[
X/(X\smallsetminus \overline{Y})\simeq X_{\mathrm{smooth}}/(X_{\mathrm{smooth}}\smallsetminus \overline{Y}).
\]
We are left to show that the image of
\[
\pi_0F^3_{\mot}K(X/(X\smallsetminus \overline{Y}))=\pi_0F^3_{\mot}K(X_{\mathrm{smooth}}/(X_{\mathrm{smooth}}\smallsetminus \overline{Y}))
\]
is generated by smooth points. Using the fact that $\overline Y$ is of codimension $2$, we obtain 
\[
\pi_0\K(X_{\mathrm{smooth}}/(X_{\mathrm{smooth}}\smallsetminus \overline{Y}))=\pi_0F_{\mot}^2\K(X_{\mathrm{smooth}}/(X_{\mathrm{smooth}}\smallsetminus \overline{Y}))
\]
and an exact sequence
\[
\pi_0F_{\mot}^3K(X_{\mathrm{smooth}}/(X_{\mathrm{smooth}}\smallsetminus \overline{Y}))\to \pi_0F_{\mot}^2K(X_{\mathrm{smooth}}/(X_{\mathrm{smooth}}\smallsetminus \overline{Y}))\to \mathrm{H}^{4,2}_{\mot}(X_{\mathrm{smooth}}/(X_{\mathrm{smooth}}\smallsetminus \overline{Y}))
\]
which reads as 
\[
\pi_0F_{\mot}^3\K(X_{\mathrm{smooth}}/(X_{\mathrm{smooth}}\smallsetminus \overline{Y}))\to \pi_0\K(X_{\mathrm{smooth}}/(X_{\mathrm{smooth}}\smallsetminus \overline{Y}))\to \mathrm{H}^{4,2}_{\mot,\overline Y}(X_{\mathrm{smooth}},\Z) \simeq \Z
\]
The middle group can be identified with the Grothendieck group of perfect complexes on $X_{\mathrm{smooth}}$ which are supported on $\overline Y$, while the right-hand homomorphism can be identified with the localization at the generic point of $\overline Y$. The image of the left-hand group is thus identified with the subgroup of perfect complexes supported on closed points of $\overline Y$, which are smooth points of $X$.
\end{proof}

\begin{rmk}
Matthew Morrow communicated to us that he has a different proof of Theorem \ref{thm:third-piece-smooth-points}, not using the normality assumption and working in any dimension $d$. The main idea is to identify the group $\H_{\mot}^{2d,d}(X,\Z)$ with the cohomology group $\H^d_{\mathrm{Nis}}(X,\mathbf{K}_d^{\mathrm M})$, and then use a result of Kato and Saito showing that the latter is generated by cycle classes of smooth closed points. This argument will appear in his ongoing work with Elden Elmanto.

We believe that the ideas developed in our proof above might be generalized to higher dimensions, under the stronger assumption that the variety is not too singular (e.g., normal).
\end{rmk}

\begin{cor}
Let $X$ be a normal affine threefold over an algebraically closed field $k$. Then, the composite 
\[
F^3_{\mathrm{mot}}\K_0(X)\xrightarrow{c_3} \mathrm{H}_{\mot}^{6,3}(X,\Z)\xrightarrow{\epsilon_3}F^3_{\mathrm{mot}}\K_0(X)
\]
is the multiplication by $2$. Furthermore, the third Chern class map $c_3$ is an isomorphism.
\end{cor}

\begin{proof}
In view of the above theorem, the group $F^3_{\mathrm{mot}}K_0(X)$ is generated by smooth points. To prove the corollary, it suffices then to show that the composite is multiplication by $2$ on these generators. Let $x\in X_{\mathrm{smooth}}$ be a smooth point. The relevant maps respect supports, and we are left to compute the composite
\[
\K_{0,x}(X)\xrightarrow{c_3} \mathrm{H}_{\mot,x}^{6,3}(X,\Z)\xrightarrow{\epsilon_3}\K_{0,x}(X)
\]
which is known to be multiplication by $2$ by Grothendieck's theorem \cite[Section 1]{Mohan82} (we may suppose that $X$ is smooth). 

We now prove that $c_3$ is an isomorphism. To start with, we observe that $F^3_{\mathrm{mot}}\K_0(X)=F^3_{\mathrm{sm}}\K_0(X)$ is uniquely divisible by \cite[Corollary 1.5]{Krishna19}. The multiplication by $2$ homomorphism is thus an isomorphism, and we conclude using Proposition \ref{prop:edgeisomorphism}.
\end{proof}

As a direct consequence of these results, we can finally prove Theorem \ref{thm:threefolds}.
\begin{proof}[Proof of Theorem \ref{thm:threefolds}]
The proof that the map is injective is the same as the one in Theorem \ref{thm:surfaces}, using the above Corollary for the injectivity of $c_3$. 

We now prove that the map is also onto. For this, we consider a triple $(a_1,a_2,a_3)$ with $a_i\in  \mathrm{H}_{\mot}^{2i,i}(X,\Z)$ for any $i=1,2,3$. Let $L$ be a line bundle with $c_1(L)=a_1$ and consider $\alpha_1=[L]-[\mathcal{O}_X]\in F^1_{\mot}\K_0(X)$. Since $c_2\colon  F^2_{\mot}\K_0(X)\to \mathrm{H}_{\mot}^{4,2}(X,\Z)$ is onto (use Lemma \ref{lem:exhaustive}), there exists $\alpha_2\in F^2_{\mot}\K_0(X)$ with $c_2(\alpha_2)=a_2$. Note that $c_1(\alpha_2)=0$ and thus that $\alpha_1+\alpha_2\in F^1_{\mot}\K_0(X)$ satisfies $c_i(\alpha_1+\alpha_2)=a_i$ for $i=1,2$. We set $b_3:=c_3(\alpha_1+\alpha_2)$ and choose $\alpha_3\in  F^3_{\mot}\K_0(X)$ with $c_3(\alpha_3)=a_3-b_3$. We claim that $\alpha:=\alpha_1+\alpha_2+\alpha_3\in F^1_{\mot}\K_0(X)$ has the required Chern classes. For this, we write $\alpha_i=[E_i]-[\mathcal{O}_X^3]$ with $E_i$ of rank $3$ for $i=2,3$ (use Serre's splitting theorem), and observe that $\alpha$ corresponds to the class of $[E_3\oplus E_2\oplus L]-[\mathcal{O}_X^7]$. Further, we have $c_2(L)=c_3(L)=0$, $c_1(E_2)=0$ and $c_1(E_3)=c_2(E_3)=0$. We can then compute the Chern classes of $E_1\oplus E_2\oplus L$ using the Whitney sum formula, which immediately yields the result. To conclude, we observe that Serre's splitting theorem shows that $[E_1\oplus E_2\oplus L]-[\mathcal{O}_X^7]=[E]-[\mathcal{O}_X^3]$ for some $E$ of rank $3$.
\end{proof}

\begin{cor}
Let $X$ be an affine normal threefold over an algebraically closed field $k$. Then the map
\[
(c_1, c_2) \colon \Vect_2(X) \to \H^{2,1}_\mot (X, \Z) \times \H^{4,2}_\mot (X, \Z)
\]
is onto. 
\end{cor}

\begin{proof}
If $(a_1,a_2)\in  \H^{2,1}_\mot (X, \Z) \times \H^{4,2}_\mot (X, \Z)$, the above theorem shows that  there exists a rank $3$ bundle $E$ with $c_i(E)=a_i$ for $i=1,2$ and $c_3(E)=0$. It remains to observe that in that case $E=E'\oplus\mathcal{O}_X$. Consider the class $\beta:=[\mathcal{O}_X]-[E]+[\Lambda^2E]-[\det E]\in \K_0(X)$. Using the splitting principle, we see that $c_1(\beta)=0$ and $c_2(\beta)=0$ (if $\ell_1,\ell_2,\ell_3$ are roots of $E$, then $\ell_1+\ell_2,\ell_1+\ell_3,\ell_2+\ell_3$ are roots of $\Lambda^2E$). Consequently, $\beta\in F^3_{\mot}\K_0(X)=F^3_{\mathrm{sm}}\K_0(X)$. Using the splitting principle once again, we see that $c_3(\beta)=-2c_3(E)=0$. Since $c_3$ is an isomorphism, we deduce that the class of $\beta$ is trivial in $F^3_{\mathrm{sm}}\K_0(X)$ and we conclude using \cite[Theorem 3.8]{Murthy94} that $E=E'\oplus\mathcal{O}_X$.
\end{proof}

\bibliographystyle{alphamod}
\let\mathbb=\mathbf
{\small
\bibliography{vectorbundles}
}

\end{document}